\documentclass[12pt]{amsart}
\usepackage{amsthm,amssymb,amsrefs,a4wide}
\newtheorem{theorem}{Theorem}[section]
\newtheorem{lemma}[theorem]{Lemma}
\newtheorem{corollary}[theorem]{Corollary} 

\newtheorem{remark}[theorem]{Remark}
\newtheorem{proposition}[theorem]{Proposition}

\newcommand{\ZZ}{\mathbb{Z}}
\newcommand{\QQ}{\mathbb{Q}}
\title{Stable range one for rings with central units}
\author{Paula A.A.B. Carvalho}
\address{Departamento de Matem\'atica, Faculdade de Ci\^encias, Universidade do Porto, Rua Campo Alegre 687, 4169-007 Porto, Portugal}
\email{pbcarval@fc.up.pt}
\author{Christian Lomp}
\address{Departamento de Matem\'atica, Faculdade de Ci\^encias, Universidade do Porto, Rua Campo Alegre 687, 4169-007 Porto, Portugal}
\email{clomp@fc.up.pt}
\author{Jerzy Matczuk}
\address{Institute of Mathematics, Warsaw University, 02-097 Warsaw, ul.Banacha 2, Poland}
\email{jmatczuk@mimuw.edu.pl}
\begin{document}

\begin{abstract}
The purpose of this paper is to give a partial positive answer to a question raised by Khurana et al. as to whether a 
ring $R$ with stable range one and central units is commutative. We show that this is the case under any of the 
following additional conditions: $R$ is semiprime or $R$ is one-sided Noetherian or $R$ has unit-stable range $1$ or $R$ 
has classical Krull dimension $0$ or $R$ is an algebra over a field $K$ such that $K$ is uncountable and $R$ has only 
countably many primitive ideals or $R$ is affine and either $K$ has characteristic $0$ or has infinite 
transcendental degree over its prime subfield or is algebraically closed. However, the general question remains open.
\end{abstract}

\maketitle

\section{Introduction}
Let $R$ be an associative, unital, but not necessarily commutative ring and denote its center by $Z(R)$ and its group of units by $U(R)$. In \cite{KhuranaMarksSrivastava}, Khurana et al. named a ring {\it unit-central} if  $U(R)\subseteq Z(R)$. It is  easy to see that in this case any element of the Jacobson radical $J(R)$ of $R$ is central as well as any nilpotent and any idempotent element of $R$.  
%

Theorem 2.3 in \cite{KhuranaMarksSrivastava} shows that any unit-central semi-exchange ring is commutative. For example any ring that is algebraic over some central subfield is strongly $\pi$-regular and hence an exchange ring by \cite[Example 2.3]{Stock}. Thus unit-central rings that are algebraic over some central subfield are commutative.

Four questions about commutativity of unit-central rings were raised in  \cite{KhuranaMarksSrivastava}. In this paper we attempt  to answer \cite[Question 3.4]{KhuranaMarksSrivastava}, which asks whether  a unit-central ring of stable range $1$ is commutative. We briefly recall the definition of stable range. Following Bass, a sequence of elements $(a_1,\ldots, a_m)\in R^m$ of a ring $R$ and $m\geq 2$ is called unimodular if $\sum_{i=1}^m Ra_i = R$. A unimodular sequence $(a_1,\ldots, a_m)$ is called stable if there exist elements $b_1, \ldots, b_{m-1}$ such that $(a_1+b_1a_m, \ldots, a_{m-1} + b_{m-1}a_m) \in R^{m-1}$ is unimodular (see \cite[\S 4]{Bass}).  The least $n\geq 1$ such that any unimodular sequence of length $m>n$ is stable is called the stable range of $R$ and denoted by $sr(R)=n$. Note that the stable range condition is left-right symmetric (see \cite{Vaserstein}). This means that a ring $R$ has stable range $1$, i.e. $sr(R)=1$, if for all $a,b\in R$ with  $Ra + Rb = R$ there exists $u\in R$ such that $a+ub$ is a unit in $R$. 

\section{Utumi's $\xi$-rings}

The next Lemma shows that for any element $x$ in a unit-central ring $R$ of stable range $1$, there exist $c(x)\in R$ such that $x^2c(x)-x$ is central. Such rings had been termed $\xi$-rings by Utumi in \cite{Utumi} and this section will be used to recall some of their properties.

\begin{lemma} Any unit-central ring of stable range $1$ is a $\xi$-ring.
\end{lemma}
\begin{proof}
For any $x\in R$, $x^2R + (1-x)R=R$ holds and by the stable range condition, there exist $c(x)\in R$ with $x^2c(x) + 1 - x \in U(R) \subseteq Z(R)$. Hence $x^2c(x)-x\in Z(R)$.
\end{proof}

Note that any division ring is trivially a $\xi$-ring, but it is unit-central if and only if it is a field. Utumi's $\xi$-rings were studied by Martindale in \cite{Martindale}, where he proved the following properties of such rings:

\begin{proposition}[Martindale]\label{Martindale_primitive} Let $R$ be a $\xi$-ring. Then:
\begin{enumerate}
\item Any $x\in R$ commutes with $c(x)\in R$ that satisfies  $x^2c(x)-x\in Z(R)$.
\item Any idempotent and any nilpotent element of $R$ is central.
\item If $R$ is left or right primitive, then it is a division ring.
\item $R$ is a subdirect product of division rings and rings whose commutators are central.
\end{enumerate}
\end{proposition}

\begin{proof}
Statements (1), (2), (3) resp. (4) follow from \cite[Theorem 1]{Martindale}, \cite[Lemma 2 and Lemma 3]{Martindale}, \cite[Theorem 2]{Martindale} resp. \cite[Main Theorem]{Martindale}. 
\end{proof}

Let $R$ be a $\xi$-ring and $x\in R$. Suppose that $x^2c(x)=x$. Then $(xc(x))^2 = xc(x)xc(x)=x^2c(x)^2=xc(x)$, since $x$ and $c(x)$ commute by Proposition \ref{Martindale_primitive}. The same Proposition says that idempotents in $\xi$-rings are central. Hence $xc(x)=c(x)x$ is central in case $x^2c(x)=x$.

We define the following function $\xi:R\rightarrow Z(R)$ for a $\xi$-ring $R$:
$$\xi (x) = \left\{ \begin{array}{ll} x & \mbox{ if } x\in Z(R)\\ xc(x) &  \mbox{ if } x\not \in Z(R) \mbox{ and } x^2c(x)=x\\  x^2c(x)-x &\mbox{otherwise}\end{array}\right.$$

The following lemma also appears in \cite[Lemma 2]{Utumi}.

\begin{lemma}[Utumi]\label{leftidealscontainideal}
Any non-zero left resp. right ideal of a $\xi$-ring $R$ contains a non-zero two-sided ideal generated by central elements. 
In particular, an element is a left zero divisor if and only if it is a right zero divisor.
\end{lemma}
\begin{proof}
Any  non-zero element $x$ in a $\xi$-ring $R$ maps to a non-zero central element $\xi(x)$. Thus if $I$ is a non-zero left ideal of $R$, then $\xi(I)$ is a non-zero subset of $I\cap Z(R)$. If $x$ is a left zero divisor, then there exists $y\in R$ with $xy=0$. Thus also $\xi(y)x=x\xi(y)=0$, i.e. $x$ is also a right zero divisor. This argument is obviously symmetric as $\xi(y)$ is non-zero and central.
\end{proof}

Recall that a ring $R$ is called subdirectly irreducible if the intersection of its non-zero two-sided ideals is non-zero. The following lemma also follows partially from \cite[page 717]{Martindale}.

\begin{lemma} \label{zerodivisors} The set of all zero divisors of a subdirectly irreducible $\xi$-ring is a two-sided ideal which is maximal as left and right ideal.
\end{lemma}

\begin{proof}
Let $H(R)$ be the intersection of the non-zero two-sided ideals of $R$. By Lemma \ref{leftidealscontainideal}, any non-zero left resp. right ideal of $R$ contains a non-zero two-sided ideal and hence contains $H(R)$. Thus $H(R)$ is an essential left and right ideal of $R$.  For any non-zero $s\in H(R)$, $\xi(s)R=R\xi(s) = H(R)$ shows that $H(R)$ is also simple as left and right ideal and generated by a non-zero central element, say $z$. Therefore, the left and right annihilators of $H(R)$ are equal to $\mathrm{ann}(H(R))=\mathrm{ann}(z)$.

Let $x\in R$ be a zero divisor. Then the left annihilator  of $x$ is non-zero and contains a non-zero ideal which contains $H(R)$; so $x\in \mathrm{ann}(H(R))$. As  any element of $\mathrm{ann}(H(R))$ is a zero divisor we conclude that  $\mathrm{ann}(H(R))$ is equal to the set of all zero divisors. Moreover, $R/\mathrm{ann}(H(R)) \simeq H(R)$ as left and right $R$-module, i.e. $\mathrm{ann}(H(R))$ is a left and right maximal ideal.
\end{proof}

\begin{proposition}\label{Paula}
Let $R$ be a subdirectly irreducible unit-central ring of stable range $1$. If all zero divisors are contained in $J(R)$, then $R$ is commutative.
\end{proposition}
\begin{proof}
As pointed out in Lemma \ref{zerodivisors}, the set of all zero divisors of $R$ is a two-sided ideal that is maximal as left and right ideal. Hence if all zero divisors are contained in $J(R)$, then $J(R)$ is the unique maximal left and right ideal of $R$ and $R$ is local. Since local rings are exchange rings, $R$ is commutative by Theorem \cite[Theorem 2.3]{KhuranaMarksSrivastava}.
\end{proof}

In general  zero divisors of a subdirectly irreducible unit-central ring of stable range $1$ do not need to be contained in $J(R)$. The example of a trivial extension $R=D\times M$, with multiplication given by $(d,m)(d',m')=(dd',dm'+md')$ for $d,d' \in D$ and $m,m'\in M$, where $D=\{\frac{a}{b}\in \QQ \mid 2,3 \nmid b\}$ and $M=\ZZ_{2^{\infty}} = E(\ZZ_2)$ (as given by Patrick Smith, see \cite{Lomp}) is a commutative semilocal subdirectly irreducible ring with exactly two maximal ideals $2D\times M$ and $3D\times M$, Jacobson radical $J(R)= 6D \times M$, set of zero divisors $2D \times M$ and essential minimal ideal $\{0\}\times \ZZ_2$.

Under some additional assumption all zero divisors of a subdirectly irreducible $\xi$-ring are central.

\begin{lemma}[Martindale]\label{zerodivisorscentral}
Let $R$ be a subdirectly irreducible ring such that for any $x\in R$ there exists a central element $c(x)$ such that $x^2c(x)-x$ is central. Then all zero divisors of $R$ are central.
\end{lemma}

\begin{proof} This follows from \cite[Theorem 4]{Martindale} and Lemma \ref{zerodivisors}. Note that Martindale uses the notation $A(S)$ for $\mathrm{ann}(H(R))$.
\end{proof}

\section{Unit-central rings of stable range $1$}
We summarize the properties of unit-central rings of stable range $1$ in the following theorem. In particular, these are PI-rings with central commutators and commutative modulo their Jacobson radical.

\begin{theorem}\label{primitive}\label{modjac}
The following statements hold for a unit-central ring $R$ of stable range $1$.
\begin{enumerate}
\item Any factor ring of $R$ is unit-central of stable range $1$.
\item  $R/J(R)$ is commutative and $J(R)$ is contained in the center.
\item  $R$ is $2$-primal, i.e. the prime radical of $R$ contains all nilpotent elements.
\item $R$ is a PI-ring; all elements $x,y,z \in R$ satisfy $[x,[y,z]]=0$ and $[x,y]^2=0$.
\item $R$ is left and right quasi-duo, i.e. any maximal one-sided ideal is two-sided.
\item  Any prime factor $R/P$ is a commutative integral domain and in particular any simple factor $R/M$ by a maximal ideal $M$ is a field and hence $M$ is a maximal left and right ideal of $R$.
\end{enumerate}
\end{theorem}

\begin{proof}
(1) Let $I$ be a two-sided ideal of $R$ and set $\overline{R}=R/I$. For any unit $x+I \in U(\overline{R})$ there exists $y\in R$ with $yx-1\in I$. Thus $Rx + I = R$ and by the stable range $1$ condition, there exist $t\in I$ and $u\in U(R)$ with $x+t=u$. Hence $x+I = u+I$. Since $R$ is unit-central, $u$ is central and hence $x+I=u+I$ is central in $R/I$. Therefore $\overline{R}$ is unit-central. 
By \cite{Vaserstein}, $\overline{R}$  has stable range $1$.

(2) By (1), any factor ring of $R$ is unit-central of stable range $1$. Hence any (left or right) primitive factor ring is a  division ring by Proposition \ref{Martindale_primitive}(3)  and since units are central it must be a field. Since $R/J(R)$ is contained in the product of its primitive factors which are commutative, also $R/J(R)$ is commutative.

(3)  It is easy and well-known that rings having all nilpotent elements central are 2-primal. 

(4) Using (2) we see that $[x,y]\in J(R)$ is central, for any $x,y$ in $R$. In particular, $R$ satisfies the identity $[[x,y],z]=0$, for all $x,y,z\in R$. Since $x[x,y]\in J(R)$ and therefore is also central we easily  get 
$[x,y]^2=0$. This gives (4).

(5)  If $M$ is any maximal left (resp. right) ideal of $R$. Then $J(R)\subseteq M$ and $M/J(R)$ is a two-sided ideal in $R/J(R)$ as $R/J(R)$ is commutative by (2). Thus,  $M$ is a two-sided ideal of $R$.

(6) By (4), any commutator $[x,y]$ is a nilpotent central element, so belongs to every prime ideal $P$ of $R$. Thus $R/P$ is commutative. In particular, for any maximal  ideal $M$, $R/M$ is commutative and hence a field.

\end{proof}

\section{Main Results}
Our main result shows that a unit-central ring of stable range $1$ is commutative under the additional assumption that it is semiprime or Noetherian or has unit $1$-stable range. Recall that a ring has {\it unit $1$-stable range} if for any $a,b \in R$ with $aR+bR=R$ there exists a unit $u\in R$, such that $a+ub$ is a unit (see \cite{GoodearlMenal}). 

\begin{theorem}\label{mainresult} \label{KrullZero} 
A unit-central ring $R$ of stable range $1$ is commutative under any of the following additional assumptions:
\begin{enumerate}
\item[(i)] $R$ is semiprime;
\item[(ii)] $R$ is left or right Noetherian;
\item[(iii)] $R$ has unit $1$-stable range;
\item[(iv)] Any prime ideal of $R$ is maximal.
\end{enumerate}
\end{theorem}


\begin{proof} 

(i) By Theorem \ref{modjac}, for any $x,y$ in $R$, the commutator $[x,y]$ is a central nilpotent element of $R$, hence equal to $0$ as $R$ is semiprime. 

(ii) Let $R$ be a subdirectly irreducible unit-central ring of stable range $1$. By Proposition \ref{Paula}, it is enough to show that  zero divisors of $R$ are contained in $J(R)$ to conclude that $R$ is commutative. Note that by Lemma \ref{leftidealscontainideal}, any left zero divisor of $R$ is also a right zero divisor and vice-versa. Suppose that $R$ is right Noetherian. Then for any zero divisor $x\in R$ consider the left multiplication $\lambda_x:R\rightarrow R$ given by $\lambda_x(y)=xy$, for any $y\in R$. Since $x$ is a zero divisor, then $\mathrm{Ker}(\lambda_x^n) \neq 0$ for any $n>0$. By Fitting's Lemma, there exists $n>0$, such that $\mathrm{Im}(\lambda_x^n)\cap \mathrm{Ker}(\lambda_x^n) = 0$. 
As seen in the proof of Lemma \ref{zerodivisors},  $H(R)$, the intersection of all non-zero ideals of $R$, is an essential and minimal left and right ideal of $R$. Hence $\mathrm{Ker}(\lambda_x^n)$ contains $H(R)$, as it is non-zero, and we conclude $x^nR = \mathrm{Im}(\lambda_x^n) = 0$. Thus $x$ is nilpotent and contained in $J(R)$. By Proposition \ref{Paula}, $R$ is commutative. 

In general, if $R$ is right Noetherian, unit-central ring of stable  range $1$, then any subdirectly irreducible factor is also a right Noetherian, unit-central ring  of stable range $1$ by Theorem \ref{modjac}(1) and thus commutative. As $R$ is a subdirect product of subdirectly irreducible factors, $R$ is commutative.

(iii) Let $R$ be an arbitrary unit-central ring of unit $1$-stable range. Then any subdirectly irreducible factor is also unit-central by Theorem \ref{modjac}(1) and also of unit $1$-stable range. If such a factor is semiprime, then by (i) it is also commutative. Hence assume $R$ to be a subdirectly irreducible non-semiprime unit-central ring of unit $1$-stable range. Then for any $x\in R$, $(1-x)R+x^2R=R$. Hence, there exists a unit $c(x)$ with $1-x + x^2c(x)$ being a unit. As units are central, $x^2c(x)-x$ and $c(x)$ are central, and by Lemma \ref{zerodivisorscentral} all zero divisors are central.
Let $x\in R$ be any element. Suppose $x^2=0$, then $x$ is central. Suppose $x^2\neq 0$, then there exist a unit $c(x)$ such that $x^2c(x)-x$ is central. For any $y\in R$, we have  $[x,y]$ is central and square-zero, by Theorem \ref{modjac}. Thus
$$2x[x,y] = x[x,y] + [x,y]x = x^2y-yx^2 = [x^2,y].$$
Since $c(x)$ and $x^2c(x)-x$ are central, we also have $[x,y] = c(x)[x^2,y]$.
Hence $[x,y]= 2c(x)x[x,y]$, i.e. 
\begin{equation}\label{equation1}(1-2c(x)x)[x,y] = 0.\end{equation}
Thus $1-2c(x)x$ is a zero divisor and therefore central. Hence $2c(x)x$ is central. Since $c(x)$ is a unit, $2x$ is central. Now we apply again the unit $1$-stable range condition to $(1-2x)R+x^2R=R$ (as $(1-2x)(1+2x)+4x^2=1$ ) to find  $c' \in U(R)$ with $1-2x + x^2c' = u$, for some unit $u$. In particular $x^2c' = u-1+2x$ is central and hence $x^2\in Z(R)$. Using additionally that $x-c(x)x^2$ and $c(x)$ are central we conclude $x$ is central. 

Thus all subdirectly irreducible factors of a unit-central ring $R$ of unit $1$-stable range are commutative and so is $R$ itself.

(iv) By Theorem \ref{modjac}, $R/N(R)$ is commutative and by hypothesis all prime ideals of $R/N(R)$ are maximal. By \cite[Theorem 3.71]{Lam_Lectures}, $R/N(R)$ is von Neumann regular and by \cite[Proposition 1.6]{Nicholson}, $R/N(R)$ is an exchange ring.  Since idempotents lift modulo nil ideals  by \cite[Theorem 21.28]{Lam},  $R$ is a semi-exchange ring. By \cite[Theorem 2.3]{KhuranaMarksSrivastava}, $R$ is commutative.
\end{proof}

\begin{remark}
Alternatively one could have proven Theorem \ref{mainresult}(ii) using  \cite[Theorem 2.1]{Deshpande}, which states that a right Noetherian ring $R$ that contains an essential minimal right ideal such any left zero divisor is also a right zero divisor is local. Together with \cite[Theorem 2.3]{KhuranaMarksSrivastava} this shows the commutativity of $R$.
\end{remark}

Note that for the proof of Theorem \ref{mainresult}(iii) one only needs that for any $a,b\in R$ with $aR+bR=R$ there exists a central element $u$  with with $a+ub$ being central and $u$ not a zero divisor.

\section*{Algebras over fields}

Goodearl and Menal gave various criteria for a ring to have unit $1$-stable range. One of them says that an algebra $R$ over an uncountable field  with only countably many right or left primitive ideals, and all right or left primitive factor rings of $R$  artinian has unit $1$-stable range (see \cite[Theorem 2.4]{GoodearlMenal}).

\begin{proposition}
Let $R$ be an algebra over an uncountable field with only countably many left or right primitive ideals. If $R$ is unit-central of stable range $1$, then $R$ is commutative.
\end{proposition}
\begin{proof}
 As a unit-central ring of stable range $1$ is a PI-ring by Theorem \ref{modjac}(5), $R$ is a PI-ring and any primitive factor ring of $R$ is Artinian. By Goodearl-Menal's result \cite[Theorem 2.4]{GoodearlMenal} , $R$ has unit $1$-stable range and by Theorem \ref{mainresult}(iii) is commutative.
\end{proof}

Recall that the classical Krull dimension $\mathcal{K}(R)$ of $R$ is defined to be the supremum of the gaps in chains of 
prime ideals. In particular $\mathcal{K}(R)=0$ if and only if all prime ideals of $R$ are maximal. For a commutative 
Noetherian domain $D$ a theorem of Bass says that the stable range $sr(D)$ is bounded from above by the (classical) 
Krull dimension plus one, i.e. $sr(D) \leq \mathcal{K}(D) +1$ (see \cite{Bass}). Suslin showed in  \cite[Theorem 
11]{Suslin} (see also \cite[Theorem 11.5.8]{McConnellRobson}) that if $D$ is an affine commutative domain $D$ over some 
field $K$ and if $K$ has infinite transcendental degree over its prime subfield, e.g. $K=\mathbb{R}$, then $sr(D) = 
\mathcal{K}(D) +1$. Hence if $sr(D)=1$, then $\mathcal{K}(D)=0$ and as $D$ is a domain,  $D$ must be a field. Therefore 
we can conclude:

\begin{corollary} Let $R$ be a unit-central ring of stable range $1$. If $R$ is a finitely generated algebra over some field $K$, such that $K$  has infinite transcendental degree over its prime subfield, then $R$ is commutative.
\end{corollary}

At the end of \cite{EstesOhm}, Estes and Ohm claimed that if a commutative domain $D$ is a finitely generated extension of a field $K$ such that the transcendence degree of $D$ over $K$ is non-zero, then $1$ is not in the stable range of $D$. They claimed that this would follow from their Proposition 7.6 and Noether Normalization. The mentioned result \cite[Proposition 7.6]{EstesOhm} of Estes and Ohm is:

\begin{proposition}[Estes-Ohm]\label{estesohmprop}
Let $D_0$ be an integrally closed domain with quotient field $F_0$, and let $D$ be the integral closure of $D_0$ in a finite separable extension $F$ of $F_0$. Then there exists an integer $n$ such that if $(a, b)$ is a unimodular sequence of $D_0$ which is stable in $D$, then $(a^n, b)$ is stable in $D_0$.
\end{proposition}

Noether Normalization (see \cite[Theorem 13.3]{Eisenbud}) says that given an affine commutative domain $D$ over a field $K$, there exists a number $m\geq 0$ (which is equal to the transcendental degree of $D$ over $K$) and a subring $D_0=K[x_1,\ldots, x_m]$ of $D$ which is a polynomial ring, such that $D$ is a finitely generated $D_0$-module. By \cite[Corollary 4.5]{Eisenbud}, $D$ is integral over $D_0$.

It is not clear to us, whether the fraction field of $D$ is always separable over the fraction field of $D_0$ in case $D$ has stable range $1$. Hence we do not know whether Noether Normalization can always be applied to Proposition \ref{estesohmprop} as claimed by Estes and Ohm. However in characteristic zero or for an algebraically closed base field (see \cite[Proposition 1.1.33]{Hulek}) we conclude:

\begin{corollary}\label{consequence_EstesOhm} Let $D$ be an affine commutative domain over a field of $K$ of characteristic zero or algebraically closed. If $D$ has stable range $1$, then $D$ is a field.
\end{corollary}

\begin{proof}
By Noether normalization, there exists a polynomial ring $D_0=K[x_1,\ldots, x_m]$ that is a subring of $D$ with $m$ being the transcendental degree of $D$ over $K$ and $D$ being finitely generated as $D_0$-module. Hence $D$ is integral over $D_0$. Note that the polynomial ring $D_0$ is integrally closed. Let $F_0$ be the fraction field of $D_0$ and let $F$ be the fraction field of $D$. Since $D$ over $D_0$ is finitely generated, $F$ is a finite extension over $F_0$. Also, since $K$ has characteristic zero, $F_0$ has characteristic zero and hence $F$ is a separable extension of $F_0$. In case $K$ is algebraically closed, $F$ is separable over $F_0$ by \cite[Proposition 1.1.33]{Hulek}. Let $\overline{D_0}$ be the integral closure of $D_0$ in $F$. Since $D$ is integral over $D_0$, $D\subseteq \overline{D_0}$. 

Suppose $m\geq 1$. Since $(x_1^2, 1-x_1)$ is a unimodular sequence in $D_0$ and as  $D$ has stable range $1$, the sequence $(x_1^2, 1-x_1)$ is stable in $\overline{D_0}$. By Proposition \ref{estesohmprop}, there exists $n\geq 1$ such that $(x_1^{2n}, 1-x_1)$ is stable in $D_0$, i.e. there exist $u\in D_0$ with $x_1^{2n} + u(1-x_1)$ being a unit in $D_0=K[x_1, \ldots, x_m]$, which is impossible by a degree argument.
Thus $m=0$ and $D_0=K$, i.e. $D$ is algebraic over $K$ and hence a field.
\end{proof}

\begin{theorem} Let $R$ be a unit-central ring of stable range $1$. If $R$ is a finitely generated algebra over a field $K$ that has characteristic $0$ or is algebraically closed, then $R$ is commutative.
\end{theorem}

\begin{proof} Any prime factor $R/P$ of $R$ is an affine commutative domain over a field $K$ of characteristic $0$ and has stable range $1$. By Corollary \ref{consequence_EstesOhm}, $R/P$ is a field. Thus any prime ideal of $R$ is maximal and by Lemma \ref{KrullZero}, $R$ is commutative.
\end{proof}

\section{Acknowledgement}
The first two named authors were partially supported by CMUP (UID/MAT/00144/2019), which is funded by FCT with national (MCTES) and European structural funds through the programs FEDER, under the partnership agreement PT2020.
A part of this work was done while the third named author visited the University of Porto. He would like to thank the University for hospitality and good working conditions.

\end{document}